\documentclass[a4paper]{amsart}

\usepackage[french]{babel}
\usepackage[utf8]{inputenc}
\usepackage[T1]{fontenc}
\usepackage{lmodern} 

\usepackage{enumerate}

\newcommand{\N}{\mathbb{N}}
\newcommand{\Z}{\mathbb{Z}}
\newcommand{\Q}{\mathbb{Q}}

\newtheorem{definition}{D\'efinition}
\newtheorem{theoreme}{Th\'eor\`eme}
\newtheorem{proposition}{Proposition}
\newtheorem{lemme}{Lemme}
\newtheorem{corollaire}{Corollaire}
\newtheorem{remarque}{Remarque}
\newtheorem*{hyp}{Hypothèse}

\usepackage{geometry}
\usepackage{url}
\usepackage{hyperref}

\begin{document}

\title{Suites récurrentes linéaires d'ordre $2$ à divisiblité forte}

\author{A. Bauval}\thanks{Anne Bauval, bauval@math.univ-toulouse.fr, IMT, UMR 5580, Universit\'e Toulouse~III}


\begin{abstract}

\noindent
\emph{On redémontre, de deux façons plus simples et tout aussi élémentaires, un résultat de \cite{H&S}, qui consistait à déterminer, parmi les suites d'entiers définies par
$$u_1=1,\quad u_2=R,\quad u_{n+2}=Pu_{n+1}-Qu_n$$
avec $P,Q,R\in\Z$, celles qui satisfont la condition de divisibilité forte :
$$\forall i,j\in\N^*\quad u_i\land u_j=\left|u_{i\land j}\right|,$$
où $\land$ désigne le plus grand commun diviseur.}

\begin{flushright}{\small Pour Papa.}\end{flushright}
\end{abstract}

\maketitle

\section{Introduction}

La célèbre suite de Fibonacci $F$, définie par la relation de récurrence linéaire d'ordre $2$
$$F_{n+2}=F_{n+1}+F_n$$
et l'initialisation $F_1=1$, $F_2=1$ (ou $F_0=0$, $F_1=1$) satisfait, entre autres propriétés remarquables, la divisibilité forte présentée en résumé introductif. Elle est le cas particulier $F=U(1,-1)$ de la suite de Lucas $U(P,Q)$ (définition \ref{defi:Lucas}), qui est à divisibilité forte si et seulement si $P\land Q=1$ (proposition~\ref{prop:DivLucas}) --  \cite{Lucas}, \cite{Dickson}.

Dans ce contexte, la valeur de la suite pour l'indice $0$ n'ayant \emph{a priori} pas grand intérêt ($k\land0=|k|$ pour tout entier $k$ ; voir cependant la section \ref{section:Remarques}), il semble naturel d'initialiser les suites à $u_1=1$ (simple convention de normalisation) et de considérer des suites qui dépendent de trois paramètres $P$, $Q$ et $R$ comme dans le résumé introductif.

Le cas $R=P$ correspond aux suites de Lucas. Pour $R\ne P$, on obtient de nouvelles suites. La stratégie de \cite{H&S}, pour trouver toutes celles qui sont à divisibilité forte, est de déterminer des conditions nécessaires sur $(P,Q,R)$ pour que l'équation de divisibilité forte soit vérifiée au moins pour $i,j\le10$. En identifiant les rares candidats qui survivent à ce test sans être des suites de Lucas, ils constatent que ces quelques suites supplémentaires sont à divisibilité forte. Plus précisément, leur théorème peut se reformuler simplement (sans s'embarrasser du fait que certaines des solutions \og exceptionnelles \fg{} sont en même temps des suites de Lucas : cf. remarque \ref{rem:PouQ=0}.b)  :

\begin{theoreme}\label{theo:Principal} Parmi les suites d'entiers définies par une récurrence linéaire d'ordre $2$ et de premier terme $u_1=1$, celles qui sont à divisibilité forte sont les suites de Lucas $U(P,Q)$ avec $P$ et $Q$ premiers entre eux et les suites de l'une des trois formes suivantes, avec $r\in\Z$ et $\varepsilon=\pm1$ :

$(1,r,\varepsilon,\varepsilon r,1,r,\dots)$ ou $(1,\varepsilon,-1,\varepsilon,-1,\dots)$ ou $(1,\varepsilon,-2,\varepsilon,1,-2\varepsilon,1,\varepsilon,\dots)$.
\end{theoreme}

Horák et Skula parviennent à ce résultat à l'issue de calculs pénibles, éliminant progressivement de nombreux candidats qui ne survivent qu'à une partie de leur test. Tout en reprenant leur idée principale, nous proposons des énoncés et démonstrations plus lisibles.

\cite{Schi} a généralisé leur théorème en remplaçant $\Z$ par l'anneau des entiers d'un corps de nombres (pour d'autres corps que $\Q$, il trouve de nouvelles familles de \og solutions exceptionnelles\fg). Nous présenterons une seconde démonstration du théorème de Horák et Skula, inspirée de la méthode de Schinzel mais dans le cas particulier de $\Z$,  ce qui nous dispense de faire appel comme lui à la notion d'entier algébrique, et même à la représentation des solutions d'une récurrence linéaire à l'aide des racines de son équation caractéristique. Nous resterons à un niveau basique d'arithmétique, n'utilisant essentiellement que les trois propriétés suivantes du pgcd :
$$q\land(bq+r)=q\land r ;$$
$$a\mid b\Leftrightarrow a\land b=|a|~;$$
$$\text{si }a\land b=1\text{ alors }a\land(bc)=a\land c.$$
Le plan de l'article est le suivant : la section \ref{section:Faits utiles} énonce quelques propriétés élémentaires de divisibilité des suites récurrentes linéaires d'ordre $2$, la section \ref{section:Lucas} les complète dans le cas particulier des suites de Lucas, la section \ref{section:CN} détermine une condition nécessaire de divisibilité forte, la section \ref{section:Theo1} utilise ce test pour établir le théorème, dont la section \ref{section:Theo2}  donne une seconde démonstration ; la section \ref{section:Remarques} conclut par quelques remarques sur diverses généralisations de propositions intermédiaires, en lien avec la notion de \og divisibilité faible\fg.

\section{Quelques faits utiles}\label{section:Faits utiles}
Le lemme \ref{lem:ValAbs+Period}, immédiat et valable pour une suite quelconque (non nécessairement donnée par une récurrence linéaire), servira à vérifier la divisibilité forte des candidats sélectionnés. Le lemme \ref{lem:Ind34} identifie deux premières conditions indispensables pour que les suites considérées soient à divisibilité forte. Le lemme \ref{lem:DivRec}, qui se démontre  facilement par récurrence, déduit de ces deux conditions quelques propriétés de ces suites.

\begin{lemme}\label{lem:ValAbs+Period}~
\begin{itemize}
\item Une suite d'entiers est à divisibilité forte si (et seulement si) sa valeur absolue l'est.
\item Pour n'importe quels entiers $s>0$ et $t$, la suite qui vaut $t$ pour les indices multiples de $s$ et qui vaut $1$ ailleurs est à divisibilité forte.
\end{itemize}
\end{lemme}

\hypertarget{hyp}{\`A partir de maintenant, on consid\`ere une suite d'entiers $u$ d\'efinie par r\'ecurrence comme dans le r\'esum\'e introductif :}
\begin{hyp}$u_1=1,\quad u_2=R,\quad u_{n+2}=Pu_{n+1}-Qu_n\quad(\forall n\in\N^*)$.
\end{hyp}

\begin{remarque}\label{rem:GeomLucas}{\rm  En g\'en\'eral, on peut retrouver les entiers $P$, $Q$ et $R$ \`a partir de $u_2$, $u_3$ et $u_4$ grâce aux équations $R=u_2$, $Q=PR-u_3$ et $P(u_3-R^2)=u_4-Ru_3$. Une exception survient lorsque $u_3-R^2=0$, c'est-à-dire $Q=R(P-R)$. Dans ce cas, la suite $u$ est géométrique ($u_n=R^{n-1}$) et de Lucas ($u=U(R,0)$).}
\end{remarque}

\begin{lemme}\label{lem:Ind34}
$$u_3\land u_4=1\Leftrightarrow(P\land Q=1\text{ et }R\land Q=1).$$
\end{lemme}

\begin{proof}$u_3=PR-Q\text{ et }u_4=Pu_3-QR$
donc
$$u_3\land u_4=1\Leftrightarrow u_3\land(QR)=1\Leftrightarrow(u_3\land Q=1\text{ et }u_3\land R=1)$$
\hfill\hfill$\Leftrightarrow((PR)\land Q=1\hbox{ et }Q\land R=1)\Leftrightarrow(P\land Q=1\text{ et }R\land Q=1).$
\end{proof}

\begin{remarque}\label{rem:PouQ=0}~
{\rm
\begin{enumerate}[a)]
\item Si $Q=0$, la condition ($R\land Q=P\land Q=1$) est \'equivalente \`a ($P=\pm1$ et $R=\pm1$).
Le cas $R=P$ fait partie des suites de Lucas. Dans le cas $R=-P$, on obtient deux suites : $(1,\varepsilon,-1,\varepsilon,-1,\dots)$ avec $\varepsilon:=R=\pm1$.
\item Si $P=0$, la condition $P\land Q=1$ est \'equivalente \`a $Q=\pm1$. Les suites obtenues sont de la forme $(1,r,\varepsilon,\varepsilon r,1,r,\dots)$, avec $r:=R\in\Z$ et $\varepsilon:=-Q=\pm1$ (la suite est de Lucas si $r=0$ mais aussi, d'apr\`es la remarque~\ref{rem:GeomLucas}, si $r^2=\varepsilon=1$).
\end{enumerate}
Toutes ces suites sont \`a divisibilit\'e forte d'apr\`es le lemme \ref{lem:ValAbs+Period}. Elles r\'eappara\^itront dans le th\'eor\`eme \ref{theo:BigTheo}. Le cas $Q=0$, pris en compte par Hor\'ak et Skula, est bri\`evement \'evacu\'e comme trivial par Schinzel.
}
\end{remarque}

\begin{lemme}\label{lem:DivRec}Si $P\land Q=1$ et $R\land Q=1$ alors (pour tout $n\in\N^*$) $u_n\land Q=1$ et $u_n\land u_{n+1}=1$.
\end{lemme}
(Sous les m\^emes hypoth\`eses, on peut \'egalement d\'emontrer que si $n$ est impair alors $u_n\land P=1$ et $u_n\land u_{n+2}=1$, mais ce ne sera pas utile.)

\section{Suites de Lucas}\label{section:Lucas}
Rappelons la définition des suites de Lucas et, parmi leurs nombreuses propriétés classiques, seulement celles dont nous aurons besoin, en laissant au lecteur la démonstration de quelques lemmes, par la méthode de son choix (récurrence, méthode matricielle, expression des suites à l'aide des racines de l'équation caractéristique...).

\begin{definition}\label{defi:Lucas}La suite de Lucas $U(P,Q)$ -- notée simplement $U$ lorsque les deux entiers $P$ et $Q$ sont fixés -- est définie par :
$$U_0=0,\quad U_1=1\quad\text{et}\quad\forall n\in\N\quad U_{n+2}=PU_{n+1}-QU_n.$$
\end{definition}

Nous exploiterons le cas particulier suivant d'une propriété générale des suites récurrentes linéaires d'ordre $2$ (\cite{Johnson}) :

\begin{lemme}\label{lem:AddU}Pour $U=U(P,Q)$ et pour toute suite $u$ comme dans l'\hyperlink{hyp}{hypothèse},
$$\forall m,n\in\N^*\quad u_{m+n}=U_mu_{n+1}-QU_{m-1}u_n.$$
En particulier, $u_{m+n}\land u_n=(U_mu_{n+1})\land u_n$.
\end{lemme}

\begin{proposition}\label{prop:DivLucas}$U(P,Q)$ est à divisibilité forte si (et seulement si) $P$ et $Q$ sont premiers entre eux.
\end{proposition}

\begin{proof}Si la suite est à divisibilité forte alors $P\land Q=1$, d'après le lemme \ref{lem:Ind34}.\\
Réciproquement, supposons que $P\land Q=1$ et démontrons la divisibilité forte de $U(P,Q)$.
D'après les lemmes \ref{lem:AddU} et \ref{lem:DivRec} :
$$\forall m,n\in\N^*\quad U_{m+n}\land U_n=(U_mU_{n+1})\land U_n=U_m\land U_n.$$
Par \href{https://fr.wikipedia.org/wiki/Anthyph%C3%A9r%C3%A8se}{anthyphérèse}, la divisibilité forte s'ensuit.
\end{proof}

\begin{lemme}\label{lem:uU} $\forall n\in\N^*\quad u_n=U_n+(R-P)U_{n-1}.$
\end{lemme}

\begin{lemme}\label{lem:RacineDouble}Si $P^2-4Q=0$ alors $U_n=n(P/2)^{n-1}$.
\end{lemme}

\section{Condition pour que $u_n\mid u_{2n}$}\label{section:CN}

L'énoncé suivant, inspiré d'un argument de Schinzel mais démontré ici bien plus rapidement, sera la clé de la preuve du théorème principal.

\begin{proposition}\label{prop:DivR-P}Si $R\land Q=P\land Q=1$ alors (pour tout $n\in\N^*$) :
$$u_n\mid u_{2n}\Rightarrow u_n\mid R-P.$$
\end{proposition}

\begin{proof}D'après les lemmes \ref{lem:AddU}, \ref{lem:DivRec} et \ref{lem:uU} :

$u_{2n}\land u_n=(U_nu_{n+1})\land u_n=U_n\land u_n=U_n\land(R-P)U_{n-1}=U_n\land(R-P)=u_n\land(R-P).$
\end{proof}

On peut déduire de cette proposition (ou démontrer directement) une condition nécessaire de divisibilité forte un peu plus contraignante que celle du lemme \ref{lem:Ind34} :

\begin{corollaire}\label{cor:Ind1a4}
$$(u_3\land u_4=1\text{ et }u_2\mid u_4)\Rightarrow(P\land Q=1\text{ et }R\mid P).$$
\end{corollaire}

\begin{proof}D'après le lemme \ref{lem:Ind34}, $u_3\land u_4=1$ implique $P\land Q=R\land Q=1$ or sous cette hypothèse, d'après la proposition ci-dessus, $u_2\mid u_4\Rightarrow R\mid P$.
\end{proof}

On déduit aussi de cette proposition le critère suivant, qui est une version plus aboutie des propositions 3.1 et 3.4 de Horák et  Skula et rend inutile leur proposition 3.2 (relative à $u_4\mid u_8$) :

\begin{corollaire}\label{cor:235}Si $R\ne0$ et $u_3\land u_4=1$ et si
$$\quad u_2\mid u_4,\quad u_3\mid u_6\quad\text{et}\quad u_5\mid u_{10}$$
alors, $f:=\frac PR$ est un entier et $f-1$ est divisible par $u_3$ et $u_5$.
\end{corollaire}

\begin{proof}Sous ces hypothèses, $R\mid P$ et (d'après la proposition \ref{prop:DivR-P}) $P-R=R(f-1)$ est divisible par $u_3$ et $u_5$. On conclut en utilisant qu'ils sont premiers avec $R$ : on a même $u_3\land P=Q\land P=1$ et $u_5\land P=(Qu_3)\land P=1$. (On pourrait d\'emontrer que $f-1$ est divisible par le produit $u_3u_5$ et même, si $u_4\mid u_8$, par $u_3\frac{u_4}Ru_5$, mais ce ne sera pas utile.)
\end{proof}

\section{Suites exceptionnelles}\label{section:Theo1}

Dans la recherche de suites à divisibilité forte, on impose aux paramètres $P$, $Q$, $R$ les conditions $P\land Q=1$ et $R\mid P$ (issues du corollaire \ref{cor:Ind1a4}) et, pour exclure le cas connu des suites de Lucas, on suppose $P\ne R$ mais aussi  $Q\ne R(P-R)$ (cf. remarque \ref{rem:GeomLucas}). Puisque $R$ divise $P$ sans lui être égal, il est non nul, ce qui permet de se placer dans les conditions du corollaire ci-dessus.

Le théorème \ref{theo:Principal} résulte alors immédiatement du suivant.

\begin{theoreme}\label{theo:BigTheo}Si $$P\ne R,\quad Q\ne R(P-R),\quad u_3\land u_4=1,\quad u_2\mid u_4,\quad u_3\mid u_6\quad\text{et}\quad u_5\mid u_{10}$$
alors ou bien $P=0$ et $Q=\pm1$, ou bien $P=-R=\pm1$ et $Q=0$ ou $1$.
\end{theoreme}

La famille des suites $(1,r,\varepsilon,\varepsilon r,1,r,\dots)$ du théorème \ref{theo:Principal} correspond au cas $P=0$, traité dans la remarque \ref{rem:PouQ=0}.b). Elles sont de période $1$, $2$ ou $4$.

Les deux suites de la forme $(1,\varepsilon,-1,\varepsilon,-1,\dots)$ correspondent au cas $P=\pm1$, $Q=0$, traité dans la remarque \ref{rem:PouQ=0}.a). Elles sont, à partir de l'indice $2$, de période $1$ ou $2$.

Les deux suites supplémentaires, de la forme $(1,\varepsilon,-2,\varepsilon,1,-2\varepsilon,1,\varepsilon,\dots)$, correspondent au cas\break$P=\pm1, Q=1$. Elles sont de période $3$ ou $6$.

Toutes sont à divisibilité forte, d'après le lemme \ref{lem:ValAbs+Period}.

\begin{proof}Par hypothèse,
$$k:=R^2-u_3=Q-R(P-R)\ne0$$
et $f:=\frac PR\ne1$. Excluons aussi le cas $f=0$ (c'est-à-dire $P=0$), déjà traité. Traitons le cas $f=-1$ puis montrons que les autres valeurs de $f$ sont impossibles. Pour cela, remarquons que
$$u_5=u_3^2-kP^2$$
et appliquons le corollaire \ref{cor:235}.

Si $f=-1$ (c'est-à-dire $P=-R$) alors $|k|P^2\le u_3^2+|u_5|\le(f-1)^2+|f-1|=6$ et $P^2$ n'est pas égal à $4$ (sinon, $u_3=R^2-k$ serait égal à $4\pm1$ et ne diviserait pas $|f-1|=2$). Par conséquent, $P^2=1$, si bien que $Q=k-2$, $u_3=1-k$ et $u_5=u_3^2-k$. Les seules solutions pour que $u_3$ et $u_5$ divisent $2$ sont $k=2$ ou $3$, c'est-à-dire $Q=0$ ou $1$.

Montrons maintenant par l'absurde que les autres valeurs de $f$ sont impossibles. Supposons donc $|f|\ge2$. Alors, $|f-1|<f^2$ et $(f-1)^2<3f^2$ donc :

$k>-1$ car $f^2>|f-1|\ge u_5\ge-kP^2=-kf^2R^2$ ;

$k\le3$ et $R^2=1$ car $kf^2R^2\le(f-1)^2+|f-1|<4f^2$.\\
Par conséquent, $1\le k<1+f^2$ et $u_5=(1-k)^2-kf^2=(k-1)(k-1-f^2)-f^2\le-f^2<-|f-1|$, d'où la contradiction.
\end{proof}

\section{Preuve par périodicité}\label{section:Theo2}
En utilisant pleinement la proposition \ref{prop:DivR-P} (plutôt que son corollaire \ref{cor:235}, limité à quelques indices), on peut déduire le théorème \ref{theo:Principal} d'une proposition moins forte (en un certain sens) que le théorème ci-dessus, mais dont la preuve est plus naturelle :

\begin{proposition}Si $P\ne R$, $Q\ne R(P-R)$, $u_3\land u_4=1$ et si, à partir d'un certain rang, $u_n\mid u_{2n}$, alors ou bien $P=0$ et $Q=\pm1$, ou bien $P=-R=\pm1$ et $Q=0$ ou $1$.
\end{proposition}

\begin{proof}Excluons les deux cas triviaux $P=0$ ou $Q=0$ (cf. remarque \ref{rem:PouQ=0}).\\
D'après la proposition \ref{prop:DivR-P} et l'hypothèse $R-P\ne0$, le couple $(u_n,u_{n+1})$ ne prend qu'un nombre fini de valeurs, donc la suite $u$ est périodique à partir d'un certain rang.\\
D'après l'hypothèse $\begin{vmatrix}u_1&u_2\\u_2&u_3\end{vmatrix}=PR-Q-R^2\ne0$, toute suite $(v_n)_{n\ge1}$ de nombres complexes vérifiant $v_{n+2}=Pv_{n+1}-Q v_n$ est combinaison linéaire des deux suites $(u_n)$ et $(u_{n+1})$, donc est également périodique à partir d'un certain rang.\\
En particulier, les deux racines $\alpha$ et $\beta$ de l'équation caractéristique $X^2-PX+Q=0$ sont distinctes (d'après le lemme~\ref{lem:RacineDouble}) et les suites géométriques associées, $(\alpha^n)$ et $(\beta^n)$, sont périodiques. Par conséquent, $|Q|=|\alpha\beta|=1$ et $|P|=|\alpha+\beta|<2$.\\
Chacun des deux entiers $P$ et $Q$ est donc égal à $1$ ou $-1$ et l'on a $R=-P$ (car $R\mid P=\pm1$ et $R\ne P$) et $Q\ne-1$ (car $u_3=PR-Q=-1-Q$ doit être non nul, pour diviser $|R-P|=2$).
\end{proof}

\section{Remarques}\label{section:Remarques}

\begin{remarque}{\rm Pour toutes les valeurs de $P,Q$, la suite de Lucas $U=U(P,Q)$ est à {\sl divisibilité faible}, c'est-à-dire que
$$\forall i,j\in\N^*\quad i\mid j\Rightarrow U_i\mid U_j.$$
En effet, la suite $U$ est le cas particulier $R-P=0$ de la suite $u$, or cette dernière vérifie la propriété suivante (dont la proposition \ref{prop:DivR-P} offrait une réciproque, sous l'hypothèse $R\land Q=P\land Q=1$) :
$$\text{si}\quad u_n\mid R-P\quad\text{alors}\quad\forall k\in\N^*\quad u_n\mid u_{kn}.$$
Cette propriété se démontre par récurrence sur $k$, en remarquant que -- d'après le lemme \ref{lem:AddU} -- si $u_n\mid R-P$ et si de plus $u_n\mid u_m$ (ou, ce qui d'après le lemme~\ref{lem:uU}, est équivalent : si de plus $u_n\mid U_m$) alors $u_n\mid u_{m+n}$. Accessoirement, on en déduit que si $R\mid P$ alors $\forall k\in\N^*$ $R\mid u_{2k}$.
}
\end{remarque}

\begin{remarque}{\rm La divisibilité faible de $U(P,Q)$, et sa divisibilité forte si $P\land Q=1$, sont classiques. Notre preuve de la proposition \ref{prop:DivLucas} (divisibilité forte) est une légère variante de celle de \cite{Bala}, qui remarque qu'elle est encore valide lorsqu'on remplace l'anneau des entiers par n'importe quel anneau -- commutatif et intègre -- \href{https://fr.wikipedia.org/wiki/Anneau_%C3%A0_PGCD}{à PGCD}). On peut faire la même remarque pour la divisibilité faible. Ces deux propriétés sont donc vérifiées par exemple dans n'importe quel anneau factoriel. Curieusement, le {\sl Fibonacci Quarterly} a accepté de publier un article (\cite{Norfleet}) qui consistait à les redémontrer dans le cas particulier de l'anneau $\Z[X]$.}
\end{remarque}

\begin{remarque}{\rm La proposition \ref{prop:DivR-P} se généralise, en supprimant l'hypothèse $R\land Q=P\land Q=1$ : si une suite $u$ comme dans l'\hyperlink{hyp}{hypothèse} est à divisibilité faible alors $\forall n\in\N^*\quad u_n\mid Q^{n-1}(R-P)$. Plus généralement, si une suite d'entiers $(u_n)_{n\ge1}$ vérifiant une relation de récurrence linéaire {\sl d'ordre quelconque\/} $k$
$$u_{n+k}=a_1u_{n+k-1}+\dots+a_ku_n$$
est à divisibilité faible, alors $\forall n\in\N^*\quad u_n\mid(a_k)^{n-1}[u_k-( a_1u_{k-1}+\dots+a_{k-2}u_2+a_{k-1}u_1)]$.

En effet, \cite{Hall} a démontré que si une suite d'entiers $(v_n)_{n\ge0}$ vérifiant cette même récurrence d'ordre $k$ est à divisibilité faible, alors $v_0$ est divisible par tout entier premier avec $a_k$ qui divise l'un des $v_n$, ce que \cite{Kimberling} a précisé en :
$$\forall n\in\N^*\quad v_n\mid(a_k)^nv_0.$$
Or en excluant le cas immédiat $a_k=0$,  l'étude des suites $u$ de l'énoncé ci-dessus (de premier terme $u_1$) se ramène à celle des suites $v$ de Hall (de premier terme $v_0$) en posant
$$v_0=u_k-( a_1u_{k-1}+\dots+a_{k-2}u_2+a_{k-1}u_1)\quad\text{et}\quad\forall n\in\N^*\quad v_n=a_ku_n.$$
}\end{remarque}

\begin{remarque}{\rm L'argument essentiel de la section \ref{section:Theo2} (la périodicité, déduite ici de la proposition \ref{prop:DivR-P}) s'étend également aux récurrences d'ordre quelconque. En effet, dès son article de 1979 (antérieur à celui de Horák et Skula), Kimberling avait déjà démontré qu'une suite d'entiers $(v_n)_{n\ge0}$ à divisibilité forte et vérifiant une relation de récurrence linéaire d'ordre quelconque
$$v_{n+k}=a_1v_{n+k-1}+\dots+a_kv_n$$
(avec $a_k\ne0$) est périodique dès que son premier terme $v_0$ est non nul. Autrement dit (par le même procédé que dans la remarque précédente) : une suite $(u_n)_{n\ge1}$ à divisibilité forte et vérifiant cette même récurrence est périodique dès que
$$u_k\ne a_1u_{k-1}+\dots+a_{k-1}u_1,$$
c'est-à-dire, dans notre cas ($k=2$, $a_1=P$, $0\ne a_k=a_2=-Q$, $u_1=1$, $u_2=R$), dès que $R\ne P$.

Même pour $k=2$, la situation est bien plus complexe lorsqu'on remplace $\Z$ par l'anneau des entiers d'un corps de nombres (\cite{Schi}) car $u$ n'est alors périodique qu'à produit près par une suite d'éléments inversibles de l'anneau.
}
\end{remarque}

\noindent\author{Anne Bauval,} bauval@math.univ-toulouse.fr\\
\address{\small Institut de Math\'ematiques de Toulouse\\
\' Equipe \' Emile Picard, UMR 5580, Universit\'e Toulouse III\\
118 Route de Narbonne, 31400 Toulouse - France\\
}


\begin{thebibliography}{999}
\bibitem[Bala, 2014]{Bala} Peter Bala, Divisibility sequences from strong divisibility sequences, \emph{OEIS}, 2014, p. 9, Proposition A.3\\
\href{https://oeis.org/A238600/a238600.pdf}{https://oeis.org/A238600/a238600.pdf}
\bibitem[Dickson, 1919]{Dickson} Leonard Eugene Dickson, \emph{History of the Theory of Numbers}, vol. 1, 1919, chap. 17, p. 393-411\\
\href{http://www.archive.org/details/historyoftheoryo01dick}{http://www.archive.org/details/historyoftheoryo01dick}
\bibitem[Hall, 1936]{Hall}Marshall Hall, Divisibility sequences of the third order, \emph{American J. Math}., vol. 58, \no 3, 1936, p. 577-584\\
\href{http://www.jstor.org/stable/2370976}{http://www.jstor.org/stable/2370976}
\bibitem[Hor\'ak et Skula, 1985]{H&S} P. Hor\'ak et L. Skula, A characterization of the second-order strong divisibility sequences, \emph{The Fibonacci Quarterly}, vol. 23, \no 2, 1985, p. 126-132\\
\href{http://www.fq.math.ca/Scanned/23-2/horak.pdf}{http://www.fq.math.ca/Scanned/23-2/horak.pdf}
\bibitem[Johnson, 2009]{Johnson}Robert C. Johnson, \emph{Fibonacci numbers and matrices}, 2009, Universit\'e de Durham, p. 40 (A.10)\\
\href{http://maths.dur.ac.uk/~dma0rcj/PED/fib.pdf}{http://maths.dur.ac.uk/~dma0rcj/PED/fib.pdf}
\bibitem[Kimberling, 1979]{Kimberling}C.  Kimberling, Strong  divisibility  sequences  and  some conjectures, \emph{The Fibonacci Quarterly}, vol. 17, 1979, p. 13-17\\
\href{http://www.fq.math.ca/Scanned/16-6/kimberling.pdf}{http://www.fq.math.ca/Scanned/16-6/kimberling.pdf}
\bibitem[Lucas, 1878]{Lucas} Édouard Lucas, Théorie des fonctions numériques simplement périodiques, {\sl Amer. J. Math.}, vol. 1, no 2, 1878, p. 184-196, 197-240, 289-321\\
\href{http://edouardlucas.free.fr/oeuvres/Theorie_des_fonctions_simplement_periodiques.pdf}{http://edouardlucas.free.fr/oeuvres/Theorie\_des\_fonctions\_simplement\_periodiques.pdf}\
\bibitem[Norfleet, 2005]{Norfleet}M. Norfleet, Characterization of second-order strong divisibility sequences of polynomials, \emph{The Fibonacci Quarterly}, vol. 43, \no 2, 2005, p. 166-169\\
\href{http://www.fq.math.ca/Papers1/43-2/paper43-2-12.pdf}{http://www.fq.math.ca/Papers1/43-2/paper43-2-12.pdf}
\bibitem[Schinzel, 1987]{Schi} Andrzej Schinzel, Second order strong divisibility sequences in an algebraic number field, \emph{Archivum Mathematicum}, vol. 23, \no 3, 1987, p. 181-186\\
\href{https://eudml.org/doc/18221}{https://eudml.org/doc/18221}
\end{thebibliography}
\end{document}